\documentclass[11pt, A4, reqno]{amsart}

\usepackage{amsmath,amssymb,amsfonts}
\usepackage{amsthm}
\usepackage[mathscr]{eucal}
\usepackage{enumitem}
\usepackage{hyperref}
\usepackage{xcolor}

\allowdisplaybreaks

\numberwithin{equation}{section}

\theoremstyle{definition}
\newtheorem{definition}{Definition}[section]
\newtheorem{example}[definition]{Example}

\theoremstyle{remark}

\theoremstyle{plain}
\newtheorem{proposition}[definition]{Proposition}
\newtheorem{theorem}[definition]{Theorem}
\newtheorem{lemma}[definition]{Lemma}
\newtheorem{result}[definition]{Result}

\newcommand{\unitdisk}{\mathbb{D}}

\newcommand{\koba}{\mathsf{k}}
\newcommand{\dkoba}{\kappa}

\newcommand{\linVarRad}{r_{\zeta}}
\newcommand{\objFun}{\phi_{\zeta}}

\newcommand{\rprt}{\mathsf{Re}}
\newcommand{\iprt}{\mathsf{Im}}
\newcommand{\distance}{\mathrm{dist}}
\newcommand{\dtb}[1]{\delta_{#1}}
\newcommand{\rsup}[1]{r^{\bullet}_{#1}}
\newcommand{\ess}[1]{S^{\bullet}_{#1}}
\newcommand{\sqdiff}{\gamma}

\newcommand*{\defeq}{\mathrel{\vcenter{\baselineskip0.5ex \lineskiplimit0pt \hbox{\scriptsize.}\hbox{\scriptsize.}}}=}
\newcommand*{\defines}{=\mathrel{\vcenter{\baselineskip0.5ex \lineskiplimit0pt \hbox{\scriptsize.}\hbox{\scriptsize.}}}}

\newcommand{\cmpl}{\mathsf{c}}

\newcommand{\bdy}{\partial}
\newcommand{\OM}{\Omega}

\newcommand{\smoo}{\mathcal{C}}

\newcommand{\bcdot}{\boldsymbol{\cdot}}

\newcommand{\C}{\mathbb{C}} 
\newcommand{\R}{\mathbb{R}}

\newcommand{\posint}{\mathbb{Z}_{+}}

\begin{document}

\title[Schwarz's lemma \& the Kobayashi metric on convex domains]{A form of Schwarz's lemma and a bound for the Kobayashi metric on convex domains}

\author{Anwoy Maitra}
\address{Department of Mathematics, Indian Institute of Science, Bangalore 560012, India}
\email{anwoymaitra@iisc.ac.in}

\thanks{This work is supported by a scholarship from the Indian Institute of Science}


\begin{abstract}
{We present a form of Schwarz's lemma for holomorphic maps between convex domains $D_1$ and $D_2$. This result provides a lower bound on
the distance between the images of relatively compact subsets of $D_1$ and the boundary of $D_2$. This is a natural improvement of an old
estimate by Bernal-Gonz{\'a}lez that takes into account the geometry of $\partial{D_1}$. Using similar techniques, we also provide a new
estimate for the Kobayashi metric on bounded convex domains.} 
\end{abstract}
\maketitle

\vspace{-0.2cm}
\section{Introduction}

In this paper, we shall prove two theorems concerning the Kobayashi geometry of convex domains in $\C^n$. In this section, we introduce
these theorems and discuss some of the motivations behind them.

\smallskip 

Our first theorem is motivated by the following result of Bernal-Gonz{\'a}lez:
\begin{result}[Bernal-Gonz{\'a}lez \cite{BG}] \label{R:BGs_result}
Let $E$, $F$ be two complex Banach spaces. Assume that $D_1 \subseteq E$ and $D_2 \subseteq F$ are convex domains and that
$D_1$ is bounded. Fix two points $a \in D_1$ and $b \in D_2$ and a real number $r > 0$. Then there exists a real number 
$\sigma = \sigma(a,b,r) > 0$ such that, for every holomorphic map $\phi: D_1 \to D_2$ satisfying $\phi(a)=b$,
\begin{equation*}
\distance\big( \phi(\{ z \in D_1 \mid \distance(z,D^{\cmpl}_1) > r \}),D^{\cmpl}_2 \big) \geqslant \sigma,
\end{equation*}
where
\begin{equation} \label{E:BGs_lower_bound}
\sigma \defeq \distance(b,D^{\cmpl}_2) \exp\left( -\frac{2 \mu(a)}{\min\big(r,\distance(a,D^{\cmpl}_1)\big)} \right),
\end{equation}
and where $\mu(a) \defeq \sup(\{ \|z-a\| \mid z \in D_1 \})$.  	
\end{result}

While this result is stated in a setting that is very general, the dependence of the lower bound for
$\distance\big( \phi(\{ z \in D_1 \mid \distance(z,D^{\cmpl}_1) > r \}), D^{\cmpl}_2 \big)$ on the parameter
$r$ seems to be overly conservative, given that
\[
\exp\left( - \frac{2 \mu(a)}{ \min(r, \distance(a,D^{\cmpl}_1)) } \right)
\]
decays extremely rapidly as $r \to 0$. Consider, in contrast, the following: if $D_1 = D_2 = \unitdisk$ (in this paper,
$\unitdisk$ will denote the open unit disk centred at $0 \in \C$), $E = F = \C$ and $a, b \in \unitdisk$, in the notation
of Result~\ref{R:BGs_result}, then it follows from the Schwarz--Pick lemma that for any holomorphic map
$\phi: \unitdisk \to \unitdisk$ such that $\phi(a) = b$, and for any $s \in (0,1)$,
\begin{equation} \label{E:dep_on_r_disk}
\distance\big( \phi(s\unitdisk),\unitdisk^{\cmpl} \big) \geqslant 
4^{-1}(1-s)\,\distance(a,\unitdisk^{\cmpl})\,\distance(b,\unitdisk^{\cmpl}),
\end{equation}
where $(1-s)$ serves as the parameter $r$ of Result~\ref{R:BGs_result}. It is natural to think that in finite dimensions,
the bound in \eqref{E:BGs_lower_bound} could be replaced by a power of $r$. Interestingly, it turns out that this
power can be arbitrarily large, even in dimension $1$\,---\,as we shall see through concrete examples. All of these
form the motivation for Theorem~\ref{T:main_theorem} below.
In what follows, $\|\bcdot\|$ will denote the Euclidean norm and the expression $\distance(x, S)$ ($S$ being a
non-empty set) will be understood in terms of this norm. With these remarks, we can state our first result.

\begin{theorem} \label{T:main_theorem}
Let $D_1$ and $D_2$ be open convex subsets of $\C^n$ and $\C^m$ respectively. Assume $D_1$ is bounded. Fix $a \in D_1$ and $b \in D_2$.
Then there exist constants $\alpha \geq 1$ and $C > 0$\,---\,where $\alpha$ depends only on $D_1$ and $C$ depends only on $D_1$ and
$a$\,---\,such that for every holomorphic map $\phi$ from $D_1$ to $D_2$ with $\phi(a)=b$, and for every $r>0$,
\begin{equation} \label{E:improved_BG_lb}
\distance\big( \phi ( \{ z \in D_1 \mid \distance(z,D_1^\cmpl) \geqslant r \} ), D_2^\cmpl \big)
\geqslant C \, \distance(b,D_2^{\cmpl}) \, r^{\alpha}.
\end{equation}
Moreover, when $\bdy D_1$ is $\smoo^2$-smooth, we can take $\alpha=1$ in \eqref{E:improved_BG_lb}.
\end{theorem}

While the dependence on $r$ of the lower bound in $\eqref{E:improved_BG_lb}$ is a power\,---\,which
improves upon the expression \eqref{E:BGs_lower_bound}\,---\,we reiterate that the exponent can \emph{in general} be
large. In Section~\ref{S:examples} we give an example in which any exponent $\alpha$ for which the bound 
\eqref{E:improved_BG_lb} holds true can be no smaller than a certain large number that is determined by the
geometry of $D_1$. Just as discussed
in \cite{BG}, we may view Theorem~\ref{T:main_theorem} as a form of Schwarz's lemma for convex domains.

\smallskip

Bernal-Gonz{\'a}lez's result relies upon a well-known estimate for the Carath{\'e}odory distance. This estimate actually holds true on
any bounded domain, whereas it is possible to provide sharper estimates on bounded convex domains. This is at the heart of our
improvement of Result~\ref{R:BGs_result}. It is more convenient to work with the Kobayashi distance. The improved estimate for the
Kobayashi distance that we shall use is due to Mercer \cite{Mercer}---see Section~\ref{S:prel_lemmas} for details. Our use of Mercer's
estimate is quite similar to its use recently in \cite{Zimmer,Bharali_Zimmer}.

\smallskip   

Before we move on to our second result, we need to introduce two pieces of notation: $D(a,r)$ will denote the open disk in $\C$ with
centre $a$ and radius $r$, and $\dkoba_D(p,\bcdot)$ will denote the Kobayashi pseudo-metric of the domain $D \subseteq \C^n$ at the point
$p \in D$.

\smallskip

It is of interest in many applications to be able to estimate $\dkoba_D(p,\bcdot)$. If nothing is assumed about $D$ beyond the fact that
it is a bounded convex open set, then the best result that seems to be available is the following one by Graham:

\begin{result}[Graham {\cite[Theorem~3]{Graham90}}, also see {\cite{Graham91}}] \label{R:Grahams_result}
Let $D \subseteq \C^n$ be a bounded convex open set. Given $p \in D$ and $\xi \in 
T_p^{(1,0)}D$,
we let $r$ denote the supremum of the radii of the disks centred at $p$, tangent to $\xi$, and included in $D$. Then
\begin{equation} \label{E:Grahams_bounds}
\frac{\|\xi\|}{2 r} \leqslant \dkoba_D(p,\xi) \leqslant \frac{\|\xi\|}{r}.
\end{equation} 	
\end{result}

We ought to clarify that the non-trivial bound in \eqref{E:Grahams_bounds} is the lower bound. The upper bound is a consequence
of the metric-decreasing property of holomorphic mappings. The upper bound in \eqref{E:Grahams_bounds} is achieved as an
equality in rare cases. For example, if we take $D = \unitdisk$ and $p$ to be any off-centre point ($p \neq 0$) then the upper
bound for $\dkoba_{\unitdisk}(p,1)$ given by Result~\ref{R:Grahams_result} is $1/(1-|p|)$, whereas the actual value of
$\dkoba_{\unitdisk}(p,1)$ is $1/(1-|p|^2)$, which is less than $1/(1-|p|)$. It is an interesting puzzle to 
find a better upper bound that can be stated (as is the case in Result~\ref{R:Grahams_result}) in terms of the positioning of
$(p,\xi)$. As the following theorem shows: an upper bound on $\dkoba_{D}(p,\xi)$ is available that is strictly smaller than that
provided by Result~\ref{R:Grahams_result} for $(p,\xi)$ that is, in a certain sense, ``generic'' (see the concluding sentence of
the following theorem). This theorem also shows that this more efficient bound is governed by one of two regimes,
\emph{both} of which do arise (see the examples in Section~\ref{S:examples}).

\begin{theorem} \label{T:improvement_Grahams_bound}
Let $D$ be a bounded convex open subset of $\C^n$. Let $p \in D$ and let $\xi \in T_p^{(1,0)}D \setminus \{0\}$. Write
\begin{align*}
D(\xi) &\defeq \big\{ z \in \C \mid p + (z/\|\xi\|)\xi \in D \big\}, \\
r^{\bullet}(p,\xi) &\defeq \sup ( \{ r > 0  \mid \exists \zeta \in D(\xi) \text{ such that } 0 \in D(\zeta,r) \subseteq D(\xi) \} ).
\end{align*}
Let
\begin{equation}
S^{\bullet}(p,\xi) \defeq \big\{ p + (z/\|\xi\|)\xi \mid z \in D(\xi), \; 0 \in \overline{D(z,r^{\bullet}(p,\xi))} \text{ and }
D(z,r^{\bullet}(p,\xi)) \subseteq D(\xi) \big\}.
\end{equation}
Also, for any $w \in (p + \C \, \xi)$, let $r(w,\xi)$ denote the supremum of the radii of the discs centred at $w$, tangent to $\xi$, and
included in $D$. 
Then
\begin{enumerate}
\item $S^{\bullet}(p,\xi)$ is a non-empty compact convex subset of $\C^n$ (indeed, of $D \cap (p+\C \, \xi)$) and there exists a unique
point $q(\xi) \in S^{\bullet}(p,\xi)$ such that
\begin{equation*}
\| q(\xi) - p \| = \distance(p,S^{\bullet}(p,\xi)).
\end{equation*}
\end{enumerate}
Write $\beta \defeq r^{\bullet}(p, \xi) - r(p, \xi)$ and $\sqdiff \defeq \|q(\xi) - p\|^2 - \beta^2$.
\begin{enumerate}[resume]
\item \label{C:mult-dim_thm_case_obj_fun_dcrsng_thrght} Suppose $(2 r(p,\xi) + \beta)\sqdiff \leqslant \beta r(p,\xi)^2$. Then
\begin{equation} \label{E:mult-dim_thm_cnclsn_case_obj_fun_dcrsng_thrght}
\dkoba_D(p,\xi)\,\leqslant\,\frac{r^{\bullet}(p,\xi)}{r^{\bullet}(p,\xi)^2 - \|q(\xi)-p\|^2} \|\xi\|. 
\end{equation}

\item \label{C:mult-dim_thm_case_obj_fun_not_dcrsng_thrght} Suppose $(2 r(p,\xi) + \beta)\sqdiff > \beta r(p,\xi)^2$. Then
\begin{equation} \label{E:mult-dim_thm_cnclsn_case_obj_fun_not_dcrsng_thrght}
\dkoba_D(p,\xi)\,\leqslant\,\frac{1}{2r(p,\xi)}\bcdot \frac{\beta^2}
{\|q(\xi)-p\|(\|q(\xi)-p\| - \sqrt{\sqdiff})} \|\xi\|. 
\end{equation}
\end{enumerate}
The upper bounds occurring above
are strictly smaller than $\|\xi\|/r(p,\xi)$ unless $q(\xi)=p$.	
\end{theorem}

The above result is also a part of the effort to provide more informative bounds for $\dkoba_{D}(p,\bcdot)$.
There are works describing the contribution of lower order terms in $1/\distance(p,D^{\cmpl})$ to asymptotic
expressions and to bounds for $\dkoba_{D}(p,\bcdot)$ when $D$ is a bounded strongly pseudoconvex domain;
see \cite{Graham75, Ma,Fu}. The description of the lower order terms in $1/\distance(p,D^{\cmpl})$ is in terms of
certain geometric invariants of the ($\smoo^2$-smooth) manifold $\bdy D$. While we merely study convex
domains $D \Subset \C^n$, we make absolutely no assumptions about the regularity of $\bdy D$, whence the latter descriptions
make no sense in general. Instead, we have the estimates of Theorem~\ref{T:improvement_Grahams_bound}. 

\smallskip

The estimates in Theorem~\ref{T:improvement_Grahams_bound} are often easy to work with (and simplify
to quite natural expressions) when a specific domain is given. Furthermore, these inequalities are sharp. Indeed, the last
sentence of Theorem~\ref{T:improvement_Grahams_bound} suggests that there are instances where the upper bound
provided by Result~\ref{R:Grahams_result} is not sharp whereas that provided by
Theorem~\ref{T:improvement_Grahams_bound} is. We provide a class of examples illustrating all
these points in Section~\ref{S:examples}. 

\smallskip    

In Section~\ref{S:prel_lemmas} we present the lemmas that will be needed to prove Theorem~\ref{T:main_theorem}.
Section~\ref{S:lemmas_planar_cvx_dmns} contains the proof of, essentially, the planar version of
Theorem~\ref{T:improvement_Grahams_bound}, from which a substantial part of Theorem~\ref{T:improvement_Grahams_bound} is derived.
Finally, Sections~\ref{S:proof_main_thm}~and~\ref{S:proof_improvement_Grahams_bound} contain, respectively, the proofs of
Theorems~\ref{T:main_theorem}~and~\ref{T:improvement_Grahams_bound}.

\medskip

\section{Examples} \label{S:examples}

We first present the example alluded to in the paragraph following Theorem~\ref{T:main_theorem}.

\begin{example}
{\em An example of a bounded convex domain $\Omega_h\Subset \C$ and a holomorphic map $\phi : \Omega_h\to \unitdisk$
where any $\alpha$ for which the bound \eqref{E:improved_BG_lb}, with $D_1 = \Omega_h$ and $D_2 = \unitdisk$, holds
true is large.}

\smallskip

\noindent{Consider, for an arbitrary (small) $h>0$ the bounded convex region
\[
\Omega_h \defeq D\big( i(1-h),1 \big) \cap D\big( -i(1-h),1 \big).
\]  
Write $C_1 \defeq \partial D\big( i(1-h),1 \big)$ and $C_2 \defeq \partial D\big( -i(1-h),1 \big)$. Then $C_1$ and $C_2$ intersect at
two points $c$ and $-c$, where
\[
c \defeq \sqrt{2h-h^2}.
\]
Observe that $\Omega_h$ has the shape of the cross-section of a lens, with vertices $\pm c \in \R$. 
Let us construct a biholomorphism $\phi$ from $\Omega_h$ to $\unitdisk$. Consider the following four functions:
\begin{align*}
f_1(z) &\defeq \frac{1}{z-c} \quad \forall \, z \in \C \setminus \{ c \}, \\
f_2(z) &\defeq -\left( z + \frac{1}{2c} \right) \quad \forall \, z \in \C, \\
f_3(z) &\defeq z^{\beta} \quad \forall \, z \text{ such that } \rprt \, z > 0, \\
f_4(z) &\defeq \frac{z-1}{z+1} \quad \forall \, z \in \C \setminus \{ -1 \},
\end{align*}
where $f_3$ is the holomorphic branch of the $\beta$-th power that maps $\R_{+}$ onto $\R_{+}$, 
\[
\beta \defeq \frac{\pi}{2A} \quad \text{and} \quad 
A \defeq \arctan \left( \frac{\sqrt{2h-h^2}}{1-h} \right).
\]
Here, $2A$ is the magnitude of the smaller angle that the circles $C_1$ and $C_2$ make with one another at both $c$ and $-c$, and also
(by conformality) the acute angle between the lines $L_1$ and $L_2$, which the circles $C_1$ and $C_2$ get mapped to, respectively,
under $f_1$,  at their point of intersection, $-\frac{1}{2 c}$. The composition
$\phi \defeq f_4 \circ f_3 \circ f_2 \circ f_1$ makes sense on $\Omega_h$,
and it follows from standard facts about M{\"o}bius transformations that it is a biholomorphism from $\Omega_h$ to $\unitdisk$. The
explicit expression for $\phi$ is given by
\[
\phi(z) = \frac{ \left( \frac{z+c}{2c(c-z)} \right)^{\beta} - 1 }{ \left( \frac{z+c}{2c(c-z)} \right)^{\beta} + 1 }.
\]

\smallskip

Now, for $t > 0$ small, consider the point $c-t$ of $\Omega_h$. A simple geometric argument shows that the distance $d(t)$ of this point
from $\Omega_h^{\cmpl}$ is
\[
d(t) = t \left( \frac{2c-t}{1 + \sqrt{t^2-2ct+1}} \right).
\]
Therefore $d(t) \approx t$ as $t \to 0$. Consider the set
\[
\Omega_{h,t} \defeq \{ z \in \Omega_h \mid \distance(z,\Omega_h^{\cmpl}) \geqslant d(t) \}.
\]
Note that $c-t$ is a boundary point of $\Omega_{h,t}$. Therefore $\phi(c-t) \in \phi \big( \Omega_{h,t} \big)$; it, too, is
a boundary point. Also,
\[
\phi(c-t) = \frac{ \left( \frac{2c-t}{2ct} \right)^{\beta} - 1 }{ \left( \frac{2c-t}{2ct} \right)^{\beta} + 1 }.
\]
The distance of $\phi(c-t)$ from $\unitdisk^{\cmpl}$ is
\[
1-\phi(c-t) = \frac{ 2 (2ct)^{\beta} }{ (2c-t)^{\beta} + (2ct)^{\beta} }.
\]
Therefore, denoting by $d'(t)$ the distance of $\phi(c-t)$ from $\unitdisk^{\cmpl}$, we see that
$d'(t) \approx t^{\beta}$ as $t \to 0$.

\smallskip 

Now, $\distance\big(\phi(\Omega_{h,t}),\unitdisk^{\cmpl} \big) \leqslant d'(t)$ and so
$\distance\big(\phi(\Omega_{h,t}),\unitdisk^{\cmpl} \big) = O(t^{\beta})$ as $t \to 0$. Let us
write
\[
  I \defeq \inf\{\alpha > 0 \mid \exists C(\alpha)>0 \text{ such that }
  		\distance\big(\phi(\Omega_{h,t}),\unitdisk^{\cmpl} \big) \geqslant C(\alpha)d(t)^\alpha
  		\text{ as $t \to 0$}\}.
\] 
Recall that $d(t) \approx t$ as $t \to 0$. Clearly, then, $I\nless \beta$. I.e., no exponent smaller than $\beta$
would suffice for the bound \eqref{E:improved_BG_lb} in this example. Finally, observe that $\beta$ can
be made arbitrarily large by taking $h$ to be sufficiently small.
\hfill $\blacktriangleleft$}
\end{example}

\begin{example}\label{Ex:balls}
{\em A family of examples for which the bound provided by Theorem~\ref{T:improvement_Grahams_bound} is sharp while the
upper bound provided by Result~\ref{R:Grahams_result} is \emph{strictly} greater.}

\smallskip

\noindent{Consider $D = \mathbb{B}^n$, the unit Euclidean ball with centre $0 \in \C^n$. Let us consider a point $p \in \mathbb{B}^n
\setminus \{0\}$. By the fact that unitary transformations are holomorphic automorphisms of $\mathbb{B}^n$, it suffices to consider
$(p,\xi)$ of the form $\big( x,(\|\xi\|,0,\ldots,0) \big)$, $x\in \mathbb{B}^n$. Let us write
$v \defeq (\|\xi\|,0,\ldots,0)$. We leave it to the reader to verify that
\begin{align*}
D(v) = D\big(-x_1,\sqrt{1-\|x'\|^2} \big),&\; \quad r^{\bullet}(x,v) = \sqrt{1-\|x'\|^2}, \\
S^{\bullet}(x,v) = \{(0,x')\},&\; \quad q(v) = (0,x'),
\end{align*}
where we write $x=(x_1,x')$. It is easy to see that the expression $(2r(x,v)+\beta)\gamma$ reduces to $0$, whence, by
item~\eqref{C:mult-dim_thm_case_obj_fun_dcrsng_thrght} of Theorem~\ref{T:improvement_Grahams_bound}, we get
the bound (note that $\|\xi\| = \|v\|$)
\begin{equation}\label{E:ball_kob_metric_estimate}
\dkoba_{\mathbb{B}^n}(x,v) \leqslant \frac{\sqrt{1-\|x'\|^2}}{1-\|x\|^2}\|\xi\|.
\end{equation}
The exact expression for $\dkoba_{\mathbb{B}^n}(p,\xi)$\,---\,see \cite[Section~3.5]{JarPfl}, for instance\,---\,is as follows:
\[
\dkoba_{\mathbb{B}^n}(p,\xi) = 
\dkoba_{\mathbb{B}^n}(x,v) = \left( \frac{\|v\|^2}{1-\|x\|^2} + \frac{|\langle v, x \rangle|^2}{(1-\|x\|^2)^2} \right)^{1/2}
= \frac{\big(1 - (\|x\|^2 - |x_1|^2)\big)^{1/2}}{1 - \|x\|^2}\|\xi\|.
\]
We see that whenever $p\neq 0$ and $\xi\in \C{p}$, the unitary transformation that maps
$\xi$ to $(\|\xi\|,0,\ldots,0)$ (and $p$ to $x$) gives $x = (c\|p\|, 0,\dots, 0)$, where $c$ is some complex
number with $|c| = 1$. Since, for any such $(p, \xi)$, $x' = 0$, it follows from \eqref{E:ball_kob_metric_estimate} that
the bound for $\dkoba_{\mathbb{B}^n}(p,\xi)$ provided by Theorem~\ref{T:improvement_Grahams_bound} is exactly
equal to $\dkoba_{\mathbb{B}^n}(p,\xi)$, whereas the upper bound
provided by Result~\ref{R:Grahams_result} is $\|\xi\|/(1-\|p\|)$, which is strictly greater. 
\hfill $\blacktriangleleft$}	
\end{example}

\begin{example}
{\em An example showing that the condition appearing in Item~\ref{C:mult-dim_thm_case_obj_fun_not_dcrsng_thrght} of 
Theorem~\ref{T:improvement_Grahams_bound} holds in simple situations.}

\noindent{Before presenting this example, we observe that the condition appearing in
Item~\ref{C:mult-dim_thm_case_obj_fun_dcrsng_thrght} of Theorem~\ref{T:improvement_Grahams_bound}
holds for the family of examples discussed in Example~\ref{Ex:balls}.
Now, consider the domain 
\begin{align*}
  D \defeq &\Big\{z \in \C \mid (|z|<1) \text{ or } \Big( 2^{-1} < \rprt(z) < 2 \text{ and } |\iprt(z)| <
  \frac{2}{\sqrt{3}}\big(
  1- 2^{-1}\rprt(z) \big)\,\Big) \Big\}.
\end{align*}
This is the unit disk together with all those $z \in \C$ such that $\rprt(z)>1/2$ and such that $z$ lies in the
angle formed by the tangent lines to the unit circle at the points $(1/2,\sqrt{3}/2)$ and
$(1/2,-\sqrt{3}/2)$. Consider the open subset $U$ of $D$ given by
\[
\left\{ z=x+iy \in \C \mid \sqrt{3}/2 < x < 1, \, |y| < 5/7\sqrt{3}\,\right\} \cap \unitdisk.
\]
We will show that, for every $z \in U$, the condition in Item~\ref{C:mult-dim_thm_case_obj_fun_not_dcrsng_thrght} of
Theorem~\ref{T:improvement_Grahams_bound} holds for the pair $(z,1)$ (and therefore, since we are in one dimension,
for any pair
$(z,\xi)$, where $\xi \in \C \setminus \{0\}$). By the symmetry of $D$ about the real axis, it suffices to
show that the condition holds for every $z \in U$ with $y \geqslant 0$. For such a $z$, $r(z,1)$ equals
the distance from $z$ to the tangent to the unit circle at $(1/2,\sqrt{3}/2)$,
which is $(2-x-\sqrt{3}y)/2$. Also, for such a $z$, 
$r^{\bullet}(z,1)=1$ and $S^{\bullet}(z,1)=\{0\}$. So, necessarily, $q(1)=0$. The quantity of our interest is
\[
(2 r(z,1) + \beta)\gamma - \beta r(z,1)^2,
\] 
which, after substituting the expressions for $\beta$ and $\gamma$, is:
\[
|q(1)-z|^2 \big( r^{\bullet}(z,1) + r(z,1) \big) - r^{\bullet}(z,1)^2 \big( r^{\bullet}(z,1) - r(z,1) \big).
\]
Substituting the actual values, we get:
\begin{align*} 
(x^2+y^2) \left( 1 + \frac{2-x-\sqrt{3}y}{2} \right) &- \left( 1 - \frac{2-x-\sqrt{3}y}{2} \right) \\
&> \frac{(3/4)(4-x-\sqrt{3}y)-x-\sqrt{3}y}{2} &&[\text{since } x > \sqrt{3}/2] \\
&=\frac{12-7x-7\sqrt{3}y}{8} > \frac{5-7\sqrt{3}y}{8} &&[\text{since } x < 1] \\
&\qquad\qquad\qquad\qquad \ \;>0 \quad &&[\text{since } y < \tfrac{5}{7\sqrt{3}}].  
\end{align*}
This shows that for every $z \in U$, the condition appearing
in Item~\ref{C:mult-dim_thm_case_obj_fun_not_dcrsng_thrght}
of Theorem~\ref{T:improvement_Grahams_bound} holds for the pair $(z,1)$.
\hfill $\blacktriangleleft$}
\end{example}
 
\medskip

\section{Preliminary lemmas} \label{S:prel_lemmas}
In order to prove Theorem~\ref{T:main_theorem}, one needs to efficiently estimate the Kobayashi distance on $D_1$. One of the most basic
estimates, which holds true on any bounded domain $\Omega$, is that, given a compact \emph{convex} subset $K$ of $\Omega$, the Kobayashi
distance $\koba_{\Omega}$ has the following upper bound:
\begin{equation}
\koba_{\Omega}(z,w) \leqslant \frac{1}{\distance(K,\Omega^{\cmpl})} \|z-w\| \quad \forall \, z,w \in K. \label{E:Kob_dist_est_cvx_cpt_set}
\end{equation}
This is essentially the estimate that is used by Bernal-Gonz{\'a}lez (he uses the Carath{\'e}odory distance, for which an analogue of
\eqref{E:Kob_dist_est_cvx_cpt_set} holds). We need a more efficient upper bound. By the nature of these estimates, this is a challenge
only close to $\bdy D_1$. Now \eqref{E:Kob_dist_est_cvx_cpt_set} arises from a comparison between the Kobayashi \emph{metrics} of
$\Omega$ and of an appropriate Euclidean ball embedded into $\Omega$. This comparison yields the following inequality:
\begin{equation}\label{E:Kob_met_est_cvx_pt}
\kappa_{\Omega}(p, \xi) \leqslant \frac{\|\xi\|}{\distance(p, \Omega^{\cmpl})}, \quad p\in \Omega, \ \xi\in T^{(1,0)}_p(\Omega).
\end{equation}
The above suggests that a more efficient estimate for $\koba_{D_1}$ could, in principle, be obtained by a comparison between the
Kobayashi metrics of $D_1$ and of (the embedded image of) some class of planar regions that are better adapted to the shape of $\bdy
D_1$. This leads us to appeal to an idea described and used by Mercer \cite{Mercer}. We consider the class of regions in $\C$ defined as
follows: for every $\alpha>1$, let $\Lambda_{\alpha}$ denote the image of $\unitdisk$ under the holomorphic mapping
\[
f \defeq z \mapsto (z+1)^{\frac{1}{\alpha}} : \{ w \in \C \mid \rprt \, w > -1 \} \to \C.
\]
For $\alpha=2$, $\Lambda_{\alpha}$ is the interior of one loop of the lemniscate. The following two results are proved in \cite[pp.
203--204]{Mercer}:

\begin{lemma}[Mercer, {\cite[Lemma~2.1]{Mercer}}] \label{L:mercer_planar_lemniscate}
Let $z_0 \in \Lambda_{\alpha}$. Then there exists a $C > 0$ such that, for all $z \in \Lambda_{\alpha}$,
\[
\koba_{\Lambda_{\alpha}}(z_0,z) \leqslant C + \frac{\alpha}{2} \log \left( \frac{1}{\distance(z,\Lambda_{\alpha}^{\cmpl})} \right). 
\]
\end{lemma}

The above lemma; a result on how the domains $\Lambda_{\alpha}$ relate to a given convex, planar domain; and a comparison between the
Kobayashi distances of $\Omega$ (as below) and of a suitable affine embedding of $\Lambda_{\alpha}$ into $\Omega$ yield the result that
we need:
\begin{lemma}[Mercer, {\cite[Proposition~2.3]{Mercer}}] \label{L:mercer_main_koba_dist_est} 
Let $\Omega \subseteq \C^n$ be a bounded convex domain, and let $z_0 \in \Omega$. Then there are constants $\alpha >1$ and $C(z_0) > 0$
such that, for every $z \in \Omega$,
\[
\koba_{\Omega}(z_0,z) \leqslant C(z_0) + \frac{\alpha}{2} \log \left( \frac{1}{\distance(z,\Omega^{\cmpl})} \right).
\]
\end{lemma}

The bound in the above lemma can be tighter if $\Omega$, in addition to the properties stated in
Lemma~\ref{L:mercer_main_koba_dist_est}, has $\smoo^2$-smooth boundary. In that case, one can carry out the
procedure hinted at prior to Lemma~\ref{L:mercer_main_koba_dist_est} with $\Lambda_{\alpha}$ replaced by the
unit disk $\unitdisk$. This argument is very classical and widely known. Its first step is the analogue of
Lemma~\ref{L:mercer_planar_lemniscate} for $\unitdisk$, which is just a direct calculation: fixing a $z_0\in \unitdisk$,
there exists a $C > 0$ such that
\[
\koba_{\unitdisk}(z_0, z) \leqslant C + \frac{1}{2} \log \left( \frac{1}{\distance(z,\unitdisk^{\cmpl})} \right)
\]
for all $z\in \unitdisk$. This leads to the classical result:

\begin{lemma}\label{L:smooth_main_koba_dist_est}
Let $\Omega \subseteq \C^n$ be a bounded convex domain whose boundary is $\smoo^2$-smooth. Let $z_0 \in \Omega$. Then there is a constant
$C(z_0) > 0$ such that, for every $z \in \Omega$,
\[
\koba_{\Omega}(z_0,z) \leqslant C(z_0) + \frac{1}{2} \log \left( \frac{1}{\distance(z,\Omega^{\cmpl})} \right).
\]
\end{lemma}

Finally, we have the following lemma that gives a useful estimate on $\koba_{\Omega}$ \emph{from below}:
\begin{lemma} \label{L:kob_dist_lb_conv}
Let $\OM \varsubsetneq \C^n$ be a convex domain. Then
\[
  \koba_{\OM}(z,w) \geqslant \frac{1}{2} \log\left( \frac{\distance(w,\OM^{\cmpl})}{\distance(z,\OM^{\cmpl})} \right)
  \quad \forall \, z,w \in \OM.
\]
\end{lemma}

\begin{proof}
	Fix $z,w \in \OM$. Choose $q \in \bdy \OM$ such that $\distance(z,\OM^{\cmpl})=\|z-q\|$. By the convexity of $\OM$,
	we may choose a $\C$-linear functional $F: \C^n \to \C$ such that
	\[
	  \OM \subseteq \{ x \in \C^n \mid \iprt(F(x-q))>0 \},
	\]
	i.e., such that
	\[
	  H \defeq \{ x \in \C^n \mid \iprt(F(x-q))=0 \}
	\]
	is a supporting hyperplane for $\OM$ at $q$. In fact, we can choose $F$ such that, for every
	$x \in \C^n$, $|\iprt(F(x-q))| =
	\distance(x,H)$. Consider the $\C$-affine function $T$ on $\C^n$
	given by
	\[
	  T(x) \defeq F(x-q) \quad \forall \, x \in \C^n.
	\]
	Then, $T$ maps $\OM$ holomorphically into the upper half plane
	$\mathbb{H}$. By the Kobayashi-distance-decreasing property of $T$ and the formula for
	the Kobayashi distance in $\mathbb{H}$,
	\[
	  \koba_{\OM}(z,w) \geqslant \koba_{\mathbb{H}}(T(z),T(w)) \geqslant \frac{1}{2} \log\left(
	  \frac{\iprt(T(w))}{\iprt(T(z))} \right) = \frac{1}{2} \log\left( \frac{\iprt(F(w-q))}{\iprt(F(z-q))} \right).
	\]
	Recall that $\iprt(F(z-q))=\distance(z,H)=\distance(z,\OM^{\cmpl})$. Furthermore, $\iprt(F(w-q))=\distance(w,H)
	\geqslant \distance(w,\OM^{\cmpl})$.
	Therefore, the sequence of inequalities above gives
	\[
	  \koba_{\OM}(z,w) \geqslant \frac{1}{2} \log\left( \frac{\distance(w,\OM^{\cmpl})}{\distance(z,\OM^{\cmpl})} \right).
	\]
	Since the points $z, w\in \OM$ were arbitrarily chosen, we have the conclusion desired.	
\end{proof}

\medskip

\section{Lemmas concerning planar convex domains} \label{S:lemmas_planar_cvx_dmns}
In this section we will state and prove a number of lemmas about planar convex domains, which will be used to prove our second result. We
abbreviate $\distance(x,\Omega^{\cmpl})$ to $\delta_{\Omega}(x)$ in this section.

\begin{lemma} \label{L:inclsn_balls_cvx_sets}
Let $\Omega \subseteq \C$ be an open convex set and let $p,\zeta \in \Omega$. Let $R(p), R(\zeta) > 0$ be such that $D(p,R(p)) \subseteq
\Omega$ and $D(\zeta,R(\zeta)) \subseteq \Omega$. Then, for every $t \in [0,1]$, $\Omega$ includes the disk in $\C$ with centre $(1-t)p +
t\zeta$ and radius $(1-t) R(p) + t R(\zeta)$.
\end{lemma}

We omit the proof because it is straightforward. The main idea behind the proof is to show that the disk described in the above lemma is
contained in the convex hull of the union of the disks $D(p, R(p))$ and $D(\zeta, R(\zeta))$. 

\begin{lemma} \label{L:mnmztn_obj_fn}
Let $\Omega \subseteq \C$ be an open convex set and let $p \in \Omega$. Suppose $\zeta \in \Omega$ is such that
\begin{equation} \label{E:where_p_is}
p \in \overline{D\big( \zeta,\dtb{\Omega}(\zeta) \big)}.
\end{equation}
Then, for every $t \in [0,1)$,
\[
p \in D \big( (1-t)p+t\zeta,(1-t)\dtb{\Omega}(p)+t\dtb{\Omega}(\zeta) \big) \subseteq \Omega.
\]
Now suppose $\dtb{\Omega}(\zeta) > \dtb{\Omega}(p)$. Let $\linVarRad(t) \defeq (1-t) \dtb{\Omega}(p) + t\dtb{\Omega}(\zeta) \; \forall \,
t \in [0,1]$, let $\alpha \defeq |\zeta-p|$ and let $\beta \defeq \dtb{\Omega}(\zeta)-\dtb{\Omega}(p)$. If we consider the mapping
\[
\objFun \defeq t \mapsto \frac{ \linVarRad(t) }{ \linVarRad(t)^2 - t^2 |\zeta - p|^2 } : [0,1) \to \R,   
\]
then

\begin{enumerate}

\item $\objFun$ is differentiable;

\item If $(\alpha^2 - \beta^2)(2 \dtb{\Omega}(p) + \beta) \leqslant \beta \dtb{\Omega}(p)^2$ then $\objFun$ is continuously extendable to
$[0,1]$ and the minimum value of $\objFun$ is \label{case_obj_fun_dcrsng_throughout_main_lemma}
\[
\frac{ \dtb{\Omega}(\zeta) }{ \dtb{\Omega}(\zeta)^2-|\zeta-p|^2 };
\]

\item If $(\alpha^2 - \beta^2)(2 \dtb{\Omega}(p) + \beta) > \beta \dtb{\Omega}(p)^2$ then $\objFun$ attains its minimum value at
\label{case_obj_fun_not_dcrsng_throughout_main_lemma}
\[
t(\zeta) \defeq \dtb{\Omega}(p) \frac{ -(\alpha^2 - \beta^2) + \alpha \sqrt{\alpha^2 - \beta^2} }{ (\alpha^2 - \beta^2)\beta }
\in (0, 1)
\]
and its minimum value is
\[
\objFun\big( t(\zeta) \big) = \frac{1}{2\dtb{\Omega}(p)} \bcdot \frac{ \beta^2 }{ \alpha \big(\alpha -  \sqrt{\alpha^2 - \beta^2}\big) }.
\]

\item Finally, the minimum value of $\objFun$ is less than $\frac{1}{\dtb{\Omega}(p)}$. \label{last_result_of_lemma}
	
\end{enumerate} 	    
\end{lemma}

\begin{proof}
In order to prove that
\[
p \in D \big( (1-t)p+t\zeta,(1-t)\dtb{\Omega}(p)+t\dtb{\Omega}(\zeta) \big) \subseteq \Omega,
\]
we first have to prove that $|p-\big( (1-t)p+t\zeta \big)| = t|\zeta - p| < (1-t)\dtb{\Omega}(p) + t\dtb{\Omega}(\zeta)$. But by the
condition \eqref{E:where_p_is}, 
\begin{equation} \label{E:where_p_is_again}
t|\zeta - p| \leqslant t\dtb{\Omega}(\zeta) < (1-t)\dtb{\Omega}(p) + t\dtb{\Omega}(\zeta) \quad \forall \, t < 1.
\end{equation} 
The inclusion statement follows from Lemma~\ref{L:inclsn_balls_cvx_sets}.

\smallskip 

Turning to $\objFun$, it is clear from \eqref{E:where_p_is_again} that it is well-defined and differentiable on $[0,1)$.

\smallskip 

Suppose first that $(\alpha^2 - \beta^2)(2 \dtb{\Omega}(p) + \beta) \leqslant \beta \dtb{\Omega}(p)^2$. Then note that necessarily
$|\zeta-p| < \dtb{\Omega}(\zeta)$. To see this, suppose $|\zeta-p|=\dtb{\Omega}(\zeta)$. Then 
\begin{align*}
(\alpha^2 - \beta^2) (2 \dtb{\Omega}(p) + \beta) &= \big( \dtb{\Omega}(\zeta)^2 - (\dtb{\Omega}(\zeta)-\dtb{\Omega}(p))^2 \big) (
\dtb{\Omega}(\zeta) + \dtb{\Omega}(p) ) \\
                                           &= 2 \dtb{\Omega}(\zeta)^2 \dtb{\Omega}(p) + \dtb{\Omega}(\zeta) \dtb{\Omega}(p)^2 -
                                           \dtb{\Omega}(p)^3 \\
                                           &> \beta \dtb{\Omega}(p)^2,  
\end{align*}
which is a contradiction. So $|\zeta-p| < \dtb{\Omega}(\zeta)$ and this shows that the expression for $\objFun(t)$ makes sense for $t \in
[0,1]$. Moreover, $\objFun$ is
differentiable on $[0,1]$ in this case. A calculation shows that
\[
\objFun'(t) = \frac{ -\beta \dtb{\Omega}(p)^2 + (\alpha^2 - \beta^2)(2 \dtb{\Omega}(p) t + \beta t^2) } { \big( \linVarRad(t)^2 - t^2
\alpha^2 \big)^2 }.
\]
If $\alpha \leqslant \beta$ then clearly $\objFun'(t) < 0$ for all $t \in [0,1]$. And if $\alpha > \beta$, then, for all $t \in [0,1)$,
\[
(\alpha^2 - \beta^2)(2 \dtb{\Omega}(p) t + \beta t^2) < (\alpha^2 - \beta^2)(2 \dtb{\Omega}(p) + \beta) \leqslant \beta
\dtb{\Omega}(p)^2, 
\]
whence
\begin{equation} \label{E:strict_ngtvty_drvtv_obj_fun_frst_case}
\objFun'(t) < 0 \quad \forall \, t \in [0,1).
\end{equation}  
So, in the case under consideration, one invariably has that $\objFun$ attains its minimum value at $1$, and the minimum value is
\[
\frac{ \dtb{\Omega}(\zeta) }{ \dtb{\Omega}(\zeta)^2 - |\zeta-p|^2 }.
\]
Furthermore, by \eqref{E:strict_ngtvty_drvtv_obj_fun_frst_case}, the minimum value above is less than $\frac{1}{\dtb{\Omega}(p)}$. 

\smallskip 

Suppose now that $(\alpha^2 - \beta^2)(2 \dtb{\Omega}(p) + \beta) > \beta \dtb{\Omega}(p)^2$. Then one necessarily has $\alpha > \beta$.
Furthermore, one expects a critical point of $\objFun$ in $(0, 1)$. In this case, by calculating the critical points of $\objFun$, we
obtain that $\objFun$ attains its minimum value at 
\[
t(\zeta) \defeq \dtb{\Omega}(p) \frac{ -(\alpha^2 - \beta^2) + \alpha \sqrt{\alpha^2 - \beta^2} }{ (\alpha^2 - \beta^2)\beta },
\]
which is a point of $(0,1)$, and that the minimum value of $\objFun$ is
\[
\objFun\big( t(\zeta) \big) = \frac{1}{2\dtb{\Omega}(p)}\bcdot
\frac{\beta^2}{ \alpha^2 - \alpha \sqrt{\alpha^2 - \beta^2} }.
\]
The following calculation shows that the minimum value above is \emph{less than} $\frac{1}{\dtb{\Omega}(p)}$:
\begin{align*}
& & \frac{ \beta^2 }{ \alpha^2 - \alpha \sqrt{\alpha^2 - \beta^2} } &< 2 & & \\
&\iff & 4 \alpha^2 (\alpha^2 - \beta^2) &< (2 \alpha^2 - \beta^2)^2 & & \text{(in this case $\alpha^2-\beta^2>0$)}\\
&\iff & 0 &< \beta^4. & &
\end{align*}  
As the last inequality is true, together with what we obtained in the other case, we get \eqref{last_result_of_lemma}.
\end{proof}

Before we state our last lemma we need to make two definitions. For $\Omega$ an open \emph{bounded} convex subset of $\C$ and for $p \in
\Omega$, we let
\begin{equation}
\rsup{\Omega}(p) \defeq \sup \big( \{ \, r > 0 \mid \exists \zeta \in \Omega \text{ such that } p \in D(\zeta,r) \subseteq \Omega \, \}
\big) \label{E:def_of_rsup}
\end{equation}
and we let
\begin{equation}
\ess{\Omega}(p) \defeq \{ \, \zeta \in \Omega \mid p \in \overline{ D( \zeta, \rsup{\Omega}(p) ) } \text{ and } D( \zeta,
\rsup{\Omega}(p) ) \subseteq \Omega \}.
\end{equation}

\begin{proposition} \label{P:improvement_Grahams_bound_one-dim_case}
Let $\Omega \subseteq \C$ be an open bounded convex set and let $p \in \Omega$. For any $\zeta \in \Omega$ such that $\dtb{\Omega}(\zeta)
> \dtb{\Omega}(p)$, let $\phi_{\zeta}$ denote the same function as in Lemma \ref{L:mnmztn_obj_fn}. Then 
\begin{enumerate}
\item $\ess{\Omega}(p)$ is a non-empty compact convex subset of $\Omega$. \label{chrctrstcs_ess_omega_dot_p}
\item\label{nearest_zeta_to_p} There exists a unique point $\zeta \in \ess{\Omega}(p)$ such that
\begin{equation} \label{E:zeta_realizes_minimum_dist}
|\zeta-p| = \distance(p,\ess{\Omega}(p)).
\end{equation}
\end{enumerate}
In the next two statements, $\zeta$ is the point in $\ess{\Omega}(p)$ introduced in \eqref{nearest_zeta_to_p}.
\begin{enumerate}[resume]
\item Suppose $\big( |\zeta-p|^2 - ( \rsup{\Omega}(p) - \dtb{\Omega}(p) )^2 \big) \big( \rsup{\Omega}(p) + \dtb{\Omega}(p) \big)
\leqslant ( \rsup{\Omega}(p)-\dtb{\Omega}(p) ) \dtb{\Omega}(p)^2$. Then \label{case_obj_fun_dcrsng_throughout}
\begin{equation} \label{E:res_case_obj_fun_dcrsng_throughout}
\dkoba_{\Omega}(p,1) \leqslant \frac{ \rsup{\Omega}(p) }{ \rsup{\Omega}(p)^2 - |\zeta-p|^2 } \leqslant \frac{1}{\dtb{\Omega}(p)},
\end{equation}
where the latter is an equality if and only if $\zeta = p$.
\item Suppose $\big( |\zeta-p|^2 - ( \rsup{\Omega}(p) - \dtb{\Omega}(p) )^2 \big) \big( \rsup{\Omega}(p) + \dtb{\Omega}(p) \big) > (
\rsup{\Omega}(p)-\dtb{\Omega}(p) ) \dtb{\Omega}(p)^2$. \label{case_obj_fun_not_dcrsng_throughout} Then, with $\alpha$ and $\beta$
denoting the same quantities as in  Lemma~\ref{L:mnmztn_obj_fn},
\begin{equation} \label{E:res_case_obj_fun_not_dcrsng_throughout}
  \dkoba_{\Omega}(p,1) \leqslant \frac{1}{2\dtb{\Omega}(p)}\bcdot
  \frac{\beta^2}{\alpha \big(\alpha - \sqrt{\alpha^2 - \beta^2}\big)} < \frac{1}{\dtb{\Omega}(p)}.
\end{equation} 
\end{enumerate} 
\end{proposition}

\begin{proof}
First we prove that $\ess{\Omega}(p)$ is non-empty. Choose an increasing sequence $(r_{\nu})_{\nu \geqslant 1}$ from the set occurring
in \eqref{E:def_of_rsup} such that $r_{\nu} > \rsup{\Omega}(p) - \tfrac{1}{\nu}$. For each $\nu$, there is a $\zeta_{\nu} \in \Omega$
such that
\[
p \in D(\zeta_{\nu},r_{\nu}) \subseteq \Omega.
\]
By the boundedness of $\Omega$, there exists a $w \in \overline{ \Omega }$ such that (without loss of generality) $\zeta_{\nu} \to w$.
Since, for each $\nu$, $|p-\zeta_{\nu}| < r_{\nu}$, therefore, by taking the limit, $|p-w| \leqslant \rsup{\Omega}(p)$, i.e., $p \in
\overline{ D(w, \rsup{\Omega}(p) ) }$. Now suppose $x \in D(w,\rsup{\Omega}(p))$. Let $\epsilon \defeq  (\rsup{\Omega}(p)-|x-w|)/2 $.
Choose $\nu' \in \posint$ such that $\rsup{\Omega}(p)-r_{\nu'}<\epsilon$. Then choose $\nu \geqslant \nu'$ such that $|w-\zeta_{\nu}| <
\epsilon$. Then,
\[
 |x-\zeta_{\nu}| \leqslant |x-w| + |w-\zeta_{\nu}| < |x-w| + \epsilon = \rsup{\Omega}(p) - \epsilon < r_{\nu'} \leqslant r_{\nu}.
\]
So, $x\in D(\zeta_{\nu}, r_{\nu})$, whence $x\in \Omega$. Therefore $D(w,\rsup{\Omega}(p)) \subseteq \Omega$. So $w \in
\ess{\Omega}(p)$, whence $\ess{\Omega}(p)$ is non-empty. Hence it makes sense to talk of points of $\ess{\Omega}(p)$ at least distance
from $p$. Also, note that if $w \in \ess{\Omega}(p)$ then $\dtb{\Omega}(w)=\rsup{\Omega}(p)$ (since $\dtb{\Omega}(w) \geqslant
\rsup{\Omega}(p)$, and if strict inequality held then the maximality of $\rsup{\Omega}(p)$ would be contradicted).

\smallskip

Now we prove the compactness of $\ess{\Omega}(p)$. Since $\ess{\Omega}(p)$ is a bounded subset of $\C$, 
it suffices to prove that it is a closed subset of $\C$. So suppose $(\zeta_{\nu})_{\nu \geqslant 1}$ is a sequence of points of
$\ess{\Omega}(p)$ converging to $w \in \C$. 
That $p \in \overline{D(w,\rsup{\Omega}(p))}$ is obvious.
Now we show that $D(w,\rsup{\Omega}(p)) \subseteq \Omega$.
Suppose $x \in D(w,\rsup{\Omega}(p))$. Let $\epsilon \defeq \rsup{\Omega}(p)-|x-w|$. Choose $\nu \in \posint$ such that $|w-\zeta_{\nu}|
< \epsilon$. Then
\[
 |x-\zeta_{\nu}| \leqslant |x-w| + |w-\zeta_{\nu}| < \rsup{\Omega}(p).
\] 
So $x \in D(\zeta_{\nu},\rsup{\Omega}(p))$, whence $x \in \Omega$. As this is true for any $x \in D(w,\rsup{\Omega}(p))$, the latter is
a subset of $\Omega$. So $w \in \ess{\Omega}(p)$ and this argument shows that $\ess{\Omega}(p)$ is a compact subset of $\Omega$.

\smallskip 

Finally we prove that $\ess{\Omega}(p)$ is convex. To this end, suppose $\zeta_1,\zeta_2 \in \ess{\Omega}(p)$ and that $t \in [0,1]$. We
want to prove that $(1-t) \zeta_1 + t \zeta_2 \in \ess{\Omega}(p)$. In order to do this we have to prove that $D\big( (1-t) \zeta_1 + t
\zeta_2,\rsup{\Omega}(p) \big) \subseteq \Omega$ and that $p \in \overline{D\big( (1-t) \zeta_1 + t \zeta_2,\rsup{\Omega}(p) \big)}$.
The first inclusion follows from Lemma \ref{L:inclsn_balls_cvx_sets}. As for the second containment,
\begin{align*}
|p - ((1-t)\zeta_1+t\zeta_2)| &= |((1-t)p+tp)-((1-t)\zeta_1+t\zeta_2)| \\
&\leqslant (1-t)|p-\zeta_1| + t |p-\zeta_2| \\
&\leqslant (1-t)\rsup{\Omega}(p) + t\rsup{\Omega}(p) = \rsup{\Omega}(p).  
\end{align*}
This shows that $(1-t) \zeta_1 + t \zeta_2 \in \ess{\Omega}(p)$, which proves that $\ess{\Omega}(p)$ is convex.
This proves \eqref{chrctrstcs_ess_omega_dot_p}.

\smallskip

We equip $\C$ with the standard Hilbert space structure from which the Euclidean norm arises. Note that the expressions in
\eqref{E:zeta_realizes_minimum_dist} are derived from the Euclidean norm. Since $\ess{\Omega}(p)$ is closed and convex, it follows from
a theorem in the elementary theory of Hilbert spaces (see \cite[Theorem~4.10]{Rudin}, for instance) that there is a unique $\zeta \in
\ess{\Omega}(p)$ such that \eqref{E:zeta_realizes_minimum_dist} holds.

\smallskip

Now suppose the condition in \eqref{case_obj_fun_dcrsng_throughout} holds. We divide the discussion into two further sub-cases:

\noindent{\textbf{Sub-case\,(a)} $\rsup{\Omega}(p) = \dtb{\Omega}(p)$.}

\noindent{In this case, $p \in \ess{\Omega}(p)$ and so $\zeta$ must be $p$. Consequently, in this case,
\[
\dkoba_{\Omega}(p,1) \leqslant \frac{1}{\dtb{\Omega}(p)} = \frac{ \rsup{\Omega}(p) }{ \rsup{\Omega}(p)^2 - |\zeta-p|^2 },
\]
where we have used the estimate \eqref{E:Kob_met_est_cvx_pt} to write the first inequality.}

\noindent{\textbf{Sub-case\,(b)} $\rsup{\Omega}(p) > \dtb{\Omega}(p)$.}

\noindent{In this case, we note that, since $\dtb{\Omega}(\zeta) = \rsup{\Omega}(p)$ and therefore $\dtb{\Omega}(\zeta) >
\dtb{\Omega}(p)$, we can appeal to Lemma~\ref{L:mnmztn_obj_fn}. By that lemma we have, for an arbitrary $t \in [0,1)$, $p \in D\big(
(1-t)p+t\zeta,r_{\zeta}(t) \big)$. Now we estimate $\dkoba_{\Omega}(p,1)$. For $t \in [0,1)$ arbitrary, consider the holomorphic function
\[
f_t \defeq z \mapsto (1-t)p + t\zeta + r_{\zeta}(t)z : \unitdisk \to \Omega.
\]
That $f_t(\unitdisk) \subseteq \Omega$ follows from Lemma~\ref{L:inclsn_balls_cvx_sets} with $R(p) = \dtb{\Omega}(p)$ and $R(\zeta) =
\dtb{\Omega}(\zeta)$. For every $z \in \unitdisk$, $f'_t(z) = r_{\zeta}(t)$. Write $z_0 \defeq t(p-\zeta)/r_{\zeta}(t) \in \unitdisk$.
Then, by the metric-decreasing property of $f_t$,
\[
\dkoba_{\Omega}(p,1) = \dkoba_{\Omega} \Big( f_t(z_0), f'_t(z_0) \frac{1}{r_{\zeta}(t)} \Big) \leqslant \dkoba_{\unitdisk} \Big( z_0,
\frac{1}{r_{\zeta}(t)} \Big)
= \frac{r_{\zeta}(t)}{r_{\zeta}(t)^2-t^2|\zeta-p|^2} = \phi_{\zeta}(t). 
\]
Minimizing the right-hand side of the inequality above with respect to $t$ tells us that the minimum of the function $\phi_{\zeta}$ is
an upper bound for $\dkoba_{\Omega}(p,1)$. Now we determine the minimum of $\phi_{\zeta}$. The condition satisfied by $\zeta$ is simply
a restatement of the condition occurring in \eqref{case_obj_fun_dcrsng_throughout_main_lemma} of Lemma~\ref{L:mnmztn_obj_fn}. So by that
lemma, the minimum value of $\phi_{\zeta}$ is
\[
\frac{ \rsup{\Omega}(p) }{ \rsup{\Omega}(p)^2 - |\zeta-p|^2 } < \frac{1}{\dtb{\Omega}(p)}
\]
and hence
\[
\dkoba_{\Omega}(p,1) \leqslant \min_{t \in [0,1]} \phi_{\zeta}(t) = \frac{ \rsup{\Omega}(p) }{ \rsup{\Omega}(p)^2 - |\zeta-p|^2 } <
\frac{1}{\dtb{\Omega}(p)}.
\]}

\noindent{Hence, the inequalities in \eqref{E:res_case_obj_fun_dcrsng_throughout} hold in either sub-case.}

\smallskip

Obviously, if $\zeta=p$ then the second inequality is an equality (because $\rsup{\Omega}(p)=\dtb{\Omega}(p)$). Suppose, conversely,
that the second inequality is an equality, and suppose, to get a contradiction, that $\zeta \neq p$. Then it must be that
$\rsup{\Omega}(p) > \dtb{\Omega}(p)$. Because, if not, then $\rsup{\Omega}(p)=\dtb{\Omega}(p)$, whence we have $\zeta=p$, as argued in
sub-case (a). This is a contradiction.  
So $\delta_{\Omega}(\zeta) = \rsup{\Omega}(p) > \dtb{\Omega}(p)$ and therefore we can consider $\phi_{\zeta}$ and appeal to
Lemma~\ref{L:mnmztn_obj_fn} to get that the minimum value of $\phi_{\zeta}$ is 
\[
\frac{ \rsup{\Omega}(p) }{ \rsup{\Omega}(p)^2 - |\zeta-p|^2 } < \frac{1}{\dtb{\Omega}(p)}.
\]
But that is a contradiction to the hypothesis, and this completes the proof of \eqref{case_obj_fun_dcrsng_throughout}.

\smallskip

Now suppose that the condition in \eqref{case_obj_fun_not_dcrsng_throughout} holds. In this case
$\rsup{\Omega}(p) > \dtb{\Omega}(p)$. The reasoning is similar to what occurs above. Therefore we can again consider
$\phi_{\zeta}$, appeal to Lemma~\ref{L:mnmztn_obj_fn}, and, reasoning as in the previous case, get that the minimum of
$\phi_{\zeta}$ is an upper bound for $\dkoba_{\Omega}(p,1)$. But in this case, since the condition satisfied by $\zeta$
is a restatement of the condition occurring in \eqref{case_obj_fun_not_dcrsng_throughout_main_lemma} of
Lemma~\ref{L:mnmztn_obj_fn}, the minimum is precisely the expression on the right-hand side of
\eqref{E:res_case_obj_fun_not_dcrsng_throughout}. That the latter is strictly smaller than $1/\dtb{\Omega}(p)$
follows from part~\eqref{last_result_of_lemma} of Lemma~\ref{L:mnmztn_obj_fn}.
\end{proof}

\medskip

\section{Proof of Theorem~\ref{T:main_theorem}} \label{S:proof_main_thm}
Before we proceed with our proof, we point out that its basic idea is inspired by the proof
of Bernal-Gonz{\'a}lez \cite{BG}, but with one significant departure. This
departure 
is the use of a refined estimate for
$\koba_{D_1}$ as discussed in Section~\ref{S:prel_lemmas}.
  
\begin{proof}
	Let
	\[
	D_1(r) \defeq \{ z \in D_1 \mid \distance(z,D_1^{\cmpl}) \geqslant r \}
	\]
	for every $r>0$ sufficiently small. Note that if $D_2=\C^m$, then the conclusion of the theorem is trivially true.
	Therefore, we suppose that $D_2 \varsubsetneq \C^m$. Then, by Lemma~\ref{L:kob_dist_lb_conv},
	\[
	  \koba_{D_2}(b,w) \geqslant \frac{1}{2} \log\left( \frac{\distance(b,D^{\cmpl}_2)}{\distance(w,D^{\cmpl}_2)} \right)
	  \quad \forall \, w \in D_2.
	\]
	Let $\phi$ be as in the statement of Theorem~\ref{T:main_theorem}. Then, for every $z \in D_1(r)$,
	\begin{equation} \label{E:lb_koba_dist_D2}
	  \koba_{D_2}(b,\phi(z)) \geqslant \frac{1}{2} \log\left(
	  \frac{\distance(b,D^{\cmpl}_2)}{\distance(\phi(z),D^{\cmpl}_2)} \right).
	\end{equation}
	Let us now suppose that $\bdy D_1$ has \emph{lower than} $\smoo^2$ regularity. In that case we have
	\begin{equation} \label{E:kob_dist_ineq_D2D1}
	  \koba_{D_2}(b,\phi(z)) \leqslant \koba_{D_1}(a,z) \leqslant C(a) + \frac{\alpha}{2}
	  \log\left( \frac{1}{\distance(z,D^{\cmpl}_1)} \right).
	\end{equation}
	The second inequality follows from Lemma~\ref{L:mercer_main_koba_dist_est}.  
	Therefore, by \eqref{E:lb_koba_dist_D2}, the above inequality, and the fact that $z \in D_1(r)$,
	\[
	  \frac{1}{2} \log\left( \frac{\distance(b,D^{\cmpl}_2)}{\distance(\phi(z),D^{\cmpl}_2)} \right)
	  \leqslant C(a)+\frac{\alpha}{2}
	  \log\left( \frac{1}{r} \right). 
	\]
	After exponentiating and a couple of computations, we get
	\begin{equation} \label{E:BG_proof_final_est}
	  \distance(\phi(z),D^{\cmpl}_2) \geqslant C \distance(b,D^{\cmpl}_2) r^{\alpha}, 
	\end{equation}
	where $C = e^{-2C(a)}$. Since $z \in D_1(r)$ was arbitrary, the above inequality completes the proof
	under the assumption that $\bdy D_1$ has
	lower than $\smoo^2$ regularity. In this case, $\alpha$, as obtained by our argument, is greater than $1$.
	
	\smallskip
	
	If $\bdy D_1$ is $\smoo^2$-smooth, then, by Lemma~\ref{L:smooth_main_koba_dist_est}, we may take $\alpha=1$
	in \eqref{E:kob_dist_ineq_D2D1}. Every subsequent step of the argument goes through, and we arrive at the
	conclusion of Theorem~\ref{T:main_theorem} with $\alpha=1$. This completes the proof. 			
\end{proof}

\medskip

\section{Proof of Theorem~\ref{T:improvement_Grahams_bound}} \label{S:proof_improvement_Grahams_bound}
\begin{proof}
For convenience we first define
\begin{equation} \label{E:basic_affine_embedding}
g \defeq z \mapsto p + (z/\|\xi\|)\xi : D(\xi) \to D.
\end{equation}
We note that
\begin{equation} \label{E:rltn_btwn_two_esses}
S^{\bullet}(p,\xi) = \Big\{ p + (z/\|\xi\|)\xi \mid z \in S(p,\xi) \Big\},
\end{equation}
where
\[
S(p,\xi) \defeq \{ z \in D(\xi) \mid 0 \in \overline{D(z,r^{\bullet}(p,\xi))} \text{ and } D(z,r^{\bullet}(p,\xi)) \subseteq D(\xi) \}.
\]
The definition of $r^{\bullet}(p,\xi)$ tells us that, in the language of Proposition~\ref{P:improvement_Grahams_bound_one-dim_case},
$r^{\bullet}(p,\xi) = \rsup{D(\xi)}(0)$. We see from the above that $S(p,\xi)$ is nothing but $\ess{D(\xi)}(0)$. Therefore, by
Proposition~\ref{P:improvement_Grahams_bound_one-dim_case}, $S(p,\xi)$ is a non-empty compact convex subset of $D(\xi)$ and there is a
unique point $z_0$ of $S(p,\xi)$ such that
\begin{equation}
|z_0| = \distance(0,S(p,\xi)).
\end{equation}
Then \eqref{E:rltn_btwn_two_esses} implies that $S^{\bullet}(p,\xi)$ is also a non-empty compact convex subset of $D \cap (p + \C \,
\xi)$. Let $q(\xi) \defeq p + z_0 \tfrac{\xi}{\|\xi\|}$. Then of course $q(\xi) \in S^{\bullet}(p,\xi)$ and
\[
\|q(\xi) - p\| = |z_0| = \distance(0,S(p,\xi)) = \distance(p,S^{\bullet}(p,\xi)).
\]
The last equality holds because $g$ preserves Euclidean distances. The uniqueness of $q(\xi)$ is also clear from the corresponding
uniqueness of $z_0$. We note that 
\[
\|q(\xi)-p\| = |z_0| \defines \alpha,
\]
that 
\[
r\big( p + (z/\|\xi\|)\xi, \xi \big) = \dtb{D(\xi)}(z) \quad \forall \, z \in D(\xi)
\]
in the notation of Proposition~\ref{P:improvement_Grahams_bound_one-dim_case}, and that
\[ 
\beta = r^{\bullet}(p,\xi) - r(p,\xi) = \rsup{D(\xi)}(0) - \dtb{D(\xi)}(0) = \dtb{D(\xi)}(z_0) - \dtb{D(\xi)}(0),
\]
where the $\alpha$ and $\beta$ above \emph{also} equal the quantities denoted by the same symbols in
Proposition~\ref{P:improvement_Grahams_bound_one-dim_case} with $\Omega \defeq D(\xi)$, $p \defeq 0$ and $\zeta \defeq
z_0$. (We note that the third equality above comes from the last sentence of the first paragraph of the proof of
Proposition~\ref{P:improvement_Grahams_bound_one-dim_case}.)

\smallskip

Now suppose that the condition in \eqref{C:mult-dim_thm_case_obj_fun_dcrsng_thrght} of Theorem~\ref{T:improvement_Grahams_bound} holds.
By the above observations, we can invoke Part~\eqref{case_obj_fun_dcrsng_throughout} of
Proposition~\ref{P:improvement_Grahams_bound_one-dim_case} to get 
\[
\dkoba_{D(\xi)}(0,1) \leqslant \frac{ \rsup{D(\xi)}(0) }{ \rsup{D(\xi)}(0)^2 - |z_0|^2 } = \frac{ r^{\bullet}(p,\xi) }{
r^{\bullet}(p,\xi)^2 - \| q(\xi)-p \|^2 }.
\]
By the metric-decreasing property of holomorphic mappings,
\[
\dkoba_D\Big( p,\frac{\xi}{\|\xi\|} \Big) = \dkoba_D( g(0), g'(0)\,1 ) \leqslant \dkoba_{D(\xi)}(0,1) \leqslant \frac{
r^{\bullet}(p,\xi) }{ r^{\bullet}(p,\xi)^2 - \| q(\xi)-p \|^2 }. 
\]
Now, using the homogeneity of $\dkoba_D(p,\bcdot)$, we get \eqref{E:mult-dim_thm_cnclsn_case_obj_fun_dcrsng_thrght}. We also note that,
by Proposition~\ref{P:improvement_Grahams_bound_one-dim_case},
\[
\frac{ \rsup{D(\xi)}(0) }{ \rsup{D(\xi)}(0)^2 - |z_0|^2 } < \frac{1}{\dtb{D(\xi)}(0)},
\]
which translates to
\[
\frac{ r^{\bullet}(p,\xi) }{ r^{\bullet}(p,\xi)^2 - \| q(\xi)-p \|^2 } < \frac{1}{r(p,\xi)}
\]
if $z_0 \neq 0$, i.e., the upper bound obtained is strictly smaller than $\|\xi\|/r(p,\xi)$ if $q(\xi) \neq p$.

\smallskip

Now suppose that the condition in \eqref{C:mult-dim_thm_case_obj_fun_not_dcrsng_thrght} of Theorem~\ref{T:improvement_Grahams_bound}
holds. This time
we can invoke Part~\eqref{case_obj_fun_not_dcrsng_throughout} of
Proposition~\ref{P:improvement_Grahams_bound_one-dim_case} and the inequality $\dkoba_D(p,\xi/\|\xi\|) \leqslant
\dkoba_{D(\xi)}(0,1)$ to get\,---\,by our observation above about the quantities $\alpha$ and
$\beta$\,---\,the bound \eqref{E:mult-dim_thm_cnclsn_case_obj_fun_not_dcrsng_thrght}. In this case the bound obtained is in fact
strictly smaller that $\|\xi\|/r(p,\xi)$. This completes the proof of the theorem.
\end{proof}
\smallskip

\section*{Acknowledgements}
I am grateful to my thesis adviser, Gautam Bharali, for suggesting some of the ideas in this work. I would also like to
thank him for his help with the writing of this paper. Furthermore, I am indebted to the anonymous referee of
an earlier version of this paper for suggesting a way to simplify one of the key steps in the proof of
Theorem~\ref{T:main_theorem}. Finally, I would like to thank Prof.~Harold Boas for his suggestions for
the writing of the introduction of this work.
   
\end{document}